\documentclass[a4paper,fleqn]{cas-sc}

\usepackage[numbers]{natbib}
\usepackage{cleveref}
\usepackage{amsthm}  
\usepackage{algorithm}
\usepackage{algpseudocode}
\usepackage{float}
\newtheorem{theorem}{Theorem}[section]      

\theoremstyle{definition}
\newtheorem{definition}[theorem]{Definition}

\theoremstyle{remark}

\def\tsc#1{\csdef{#1}{\textsc{\lowercase{#1}}\xspace}}
\tsc{WGM}
\tsc{QE}

\begin{document}
\let\WriteBookmarks\relax
\def\floatpagepagefraction{1}
\def\textpagefraction{.001}

\shorttitle{Direct Solver for Scaled Partial-Isometric Linear Systems}

\shortauthors{S. Kar, M. Venkatapathi}  

\title[mode=title]{An Optimal Least-Square Solver For Scaled Partial-Isometric Linear Systems}  

\tnotemark[1] 

\tnotetext[1]{This work was funded by ......} 

%

\author[1]{Suvendu Kar}


\fnmark[1]

\ead{sk73191491@gmail.com}

\ead[url]{}

\credit{}

\affiliation[1]{organization={Department of Computational and Data Sciences},
            addressline={Indian Institute of Science }, 
            city={Bangalore},
            postcode={560012}, 
            state={karnataka},
            country={India}}

\author[1]{Murugesan Venkatapathi}

\cormark[1]
\fnmark[2]

\ead{murugesh@iisc.ac.in}

\ead[url]{}

\credit{}


\cortext[1]{Corresponding author}

\fntext[1]{}


\begin{abstract}
We present an $O(mn)$ direct least-squares solver for $m \times n$ linear systems with a scaled partial isometry. The proposed algorithm is also useful when the system is block diagonal and each block is a scaled partial isometry with distinct scaling factors. We also include numerical experiments as a demonstration.
\end{abstract}



\begin{keywords}
 \sep Direct solver for linear systems \sep Scaled partial isometries  \sep Least-square solution \sep \vspace{0.5em}
\end{keywords}

\maketitle

\section{Introduction}\label{sec:introduction}

Obtaining a solution to a linear system of equations is a central task in numerical linear algebra and serves as a foundational component across a broad spectrum of quantitative disciplines. These systems, often expressed in the canonical form $Ax=b$, arise in numerous scientific, engineering, and data-driven applications, including but not limited to computational physics, signal and image processing, optimization, computer vision, machine learning, and numerical simulations. While solving a linear least-squares problem by computing the Moore-Penrose pseudoinverse \cite{svd} has a computational cost that scales cubically, the system in-hand may enjoy a special structure (for example, a system with a partial isometry) that can be utilized to evaluate the least-squares solution at a significantly lower computational cost.

A matrix $A \in \mathbb{C}^{m \times n}$ is said to be a partial isometry if $AA^*A=A$ \cite{partialisometryarxiv}, $A^*$ being the conjugate-transpose of $A$. We present an $O(mn)$ least-squares direct linear solver for the linear system $Ax=b$ where $A$ can be a \emph{scaled} partial isometry. Partial isometries represent a significant and appealing class of operators. An operator is called a partial isometry if it acts as an isometry when restricted to the orthogonal complement of its null space. These operators are fundamental in the field of operator theory; for example, they are essential in the polar decomposition of general operators and serve as a foundation for the theory of von Neumann algebras \cite{partialisometryelsiver}. Common instances of partial isometries include orthogonal projections, unitary operators, isometries, and co-isometries \cite{partialisometryelsiver}. Partial isometries describe certain quantum measurements and state transformations, especially when dealing with non-orthogonal states or incomplete measurements \cite{partialisometryieee}. In signal processing and frame theory, they are used in the construction of tight frames, filter banks, and optimal measurement vectors, where partial isometries ensure energy-preserving transformations on subspaces \cite{partialisometryieee}. In finite dimensions, partial isometries are used to study matrix factorizations, similarity, and unitary equivalence problems \cite{partialisometryarxiv}.

The paper is organized as follows: \cref{subsec:notation} provides the details of the notations used in this paper. A few useful properties of partial isometries are outlined in \cref{sec:properties}. \Cref{sec:methods} contains the proposed algorithm, its applicability, and its theoretical analysis. Experiments were conducted to validate the proposed algorithm, and the corresponding results can be seen in \cref{sec:numerical_experiments}. Finally, we conclude the paper with some remarks in \cref{sec:conclusion}.

\subsection{Notations Used}\label{subsec:notation}
The symbol $A^{\dagger}$ denotes the Moore–Penrose pseudo-inverse \cite{svd} of matrix $A \in \mathbb{C}^{m \times n}$. $A^*$ denotes the conjugate-transpose of $A$. $A^{(i)}, A_{(i)}$ denotes the $i^{\text{th}}$ column and row of the matrix, $A$ respectively. Unless specifically mentioned, $\|.\|$ as well as $\|.\|_2$ denote 2-norm, and $\|.\|_F$ denotes Frobenius norm. For two compatible column vectors $v_1,v_2$, we define the inner product as $\langle v_1, v_2\rangle = v_1^*v_2$.

\section{Properties of Partial Isometries}\label{sec:properties}

\begin{definition}
    A matrix $A \in \mathbb{C}^{m \times n}$ is siad to be a partial isometry if $AA^*A=A$ \cite{partialisometryarxiv}.
\end{definition}

While the above definition is brief, it offers limited insight into the concept of a partial isometry. Nevertheless, it does imply connections to unitary matrices, orthogonal projections, and the Moore–Penrose pseudoinverse—associations that are both meaningful and worth exploring

\begin{theorem}\label{thm:thm1}
    A matrix $A \in \mathbb{C}^{m \times n} $ is a partial isometry iff all its non-zero singular values (if any) are 1.
\end{theorem}
\begin{proof}
    ($\Rightarrow$) To show that if $A$ is a partial-isometry then all its non-zero singular values (if any) are 1.
    
    If $A$ is a zero matrix, then the claim trivially holds. Let's say, rank of $A$ be $r > 1$. Then the partial singular-value decomposition of $A$ can written as $A=U\Sigma V^*$, where $U \in \mathbb{C}^{m \times r}, \Sigma \in \mathbb{R}^{r \times r}, \text{and} \quad V \in \mathbb{C}^{n \times r}$. Now $A$ being a partial isometry: $AA^*A=A$ holds and thus 
    \begin{align*}
        &(U\Sigma V^*)(U\Sigma V^*)^*(U\Sigma V^*)=(U\Sigma V^*)\\
        &\implies U\Sigma^3V^* =U\Sigma V^*\\
        &\implies U^*(U\Sigma^3V^*)V = U^*(U\Sigma V^*)V\\
        &\implies \Sigma^3=\Sigma
    \end{align*}
    Now $\Sigma \in \mathbb{R}^{r \times r}$ being a full-rank diagonal matrix, from above it is clear that the diagonal entries of $\Sigma$ are all 1, as by convention non-zero singular values are positive. And hence all the non-zero singular values of $A$ in this case are 1.

    ($\Leftarrow$) To show that if all the non-zero singular values (if any) of $A$ are 1, then it is a partial isometry.

    If $A$ has no non-zero singular values, then trivially it is a partial isometry. Now, let's say $A$ has $r>1$ non-zero singular values, all of which are 1. Then, using partial singular-value decomposition $A=U\Sigma V^*$, where $\Sigma =I_r$, an $r \times r$ identity matrix. Thus, $A=UV^*$. Now, $AA^*A=(UV^*)(UV^*)^*(UV^*)=UV^*=A$. Thus, A is a partial isometry.
\end{proof}

\begin{theorem}\label{thm:thm2}
    For $A \in \mathbb{C}^{m \times n} $, the following statements has relation $(a) \implies (b)$ and $(a) \implies (c)$.\\
 (a)$A$ is a partial isometry\\
 (b)$A=WP$, where W is a partial isometry and P is an orthogonal projection.\\ 
 (c)$A=QW$, where W is a partial isometry and Q is an orthogonal projection.
\end{theorem}
\begin{proof}
    If $A$ is a zero-matrix, then the proofs for both the relations $(a) \implies (b)$ and $(a) \implies (c)$ are trivial.
    
    To show $(a) \implies (b)$ :\\
     Let $A=U \Sigma V^*$ be a partial singular-value decomposition of the non-zero partial isometry $A$. Then $A=(UV^*)(V\Sigma V^*)$, and it is easy to see that $W=UV^*$ is a partial isometry. Following \cref{thm:thm1} $A$ being a partial isometry $\Sigma$ is an identity matrix and thus $P=V\Sigma V^*=VV^*$ is an orthogonal projection.

    To show $(a) \implies (c)$ :\\
    Similarly, using the partial SVD of $A$, we have $A= U \Sigma V^* = (UU^*)(U\Sigma V^*)$. It is trivial that $Q=UU^*$ is an orthogonal projection and $W=U\Sigma V^*$ is a partial isometry (using the information of $\Sigma$ from \cref{thm:thm1}).
    \end{proof}

\cite{partialisometryarxiv, partialisometryelsiver,partialisometrynote,f1,f3,f5,f6,a1} can be referred for more details about partial isometries.

\section{Solution of a Scaled Partial-Isometric Linear System}\label{sec:methods}
Here we outline the $O(mn)$ procedure to evaluate the minimum 2-norm solution of a system $Ax=b$ where A is a scaled partial isometry.

\begin{theorem}\label{thm:direct_solver}
For a matrix $A \in \mathbb{C}^{m\times n}$ with at least one non-zero singular value and all non-zero singular values being the same, the least square solution of the system $Ax=b$ is given by: 
\begin{equation}
x_i =
\begin{cases}
\frac{(\langle A^{(i)}, b\rangle)^*(\langle A^{(i)}, b\rangle)}{(A^{*}b)^{*}A^{*}A^{(i)}} & \text{if } b^{*}AA^{*}A^{(i)} \neq 0 \\
0 & \text{otherwise}
\end{cases}
\label{eq:a1}
\end{equation}
\end{theorem}
\begin{proof}
The case where A does not have a non-zero singular value is trivial.

Let the partial singular value decomposition of A be $A=U\Sigma V^{*}$, where $U \in \mathbb{C}^{m\times r}, \Sigma \in \mathbb{R}^{r\times r}, V \in \mathbb{C}^{n\times r}$. Let $\alpha$ be the non-zero singular values. Then the least square solution of the system of equations $Ax=b$ is given as $x= V\Sigma ^{-1}U^{*}b$ and thus 
\begin{equation}
x_i= \langle e^{(i)},x\rangle=e^{(i)^*}x=V_{(i)}\Sigma ^{-1}U^{*}b=\frac{1}{\alpha}V_{(i)}U^{*}b
\label{eq:a2}
\end{equation}
[ $e^{(i)}$ being an $n$-dimensional canonical-basis vector with the $i^{th}$ entry as 1 and all other entries as 0.]\\

Now, \begin{align}
\label{eq:a3.1}
& A^{(i)}=U\Sigma V_{(i)}^{*}\\ 
\label{eq:a3.2}
 & A^{(i)^{*}}b=V_{(i)}\Sigma U^{*}b\\
 \label{eq:a3.3}
 & A^{*}b=V\Sigma U^{*}b\\
 \label{eq:a3.4}
 & A^{*}A^{(i)}=V \Sigma ^{2}V_{(i)}^{*}
\end{align}

The condition $b^{*}AA^{*}A^{(i)} = 0$ $ \implies b^{*}U\Sigma V^{*}V\Sigma U^{*}U\Sigma V_{(i)}^{*} =0$, using the relations \eqref{eq:a3.1}- \eqref{eq:a3.4}.

\begin{align*}
 & \implies b^{*}U\Sigma ^{3} V_{(i)}^{*} =0 \\
& \implies \alpha ^3 b^*UV_{(i)}^*=0 \\
& \implies \alpha ^4 \frac{1}{\alpha}V_{(i)}U^{*}b=0\\
& \implies \alpha^4 x_i = 0, \text{ using } \eqref{eq:a2},
\end{align*}
which validates our formulation of considering $x_i=0$ when $b^{*}AA^{*}A^{(i)} = 0$ and $\alpha \neq 0$ in $\eqref{eq:a1}$.

Now, when $b^{*}AA^{*}A^{(i)} \neq 0$, the proposed solution
\begin{align*}
    x_i &= \frac{(\langle A^{(i)}, b\rangle)^*(\langle A^{(i)}, b\rangle)}{(A^{*}b)^{*}A^{*}A^{(i)}} \text{ where using relations (3)-(6) above,}\\
    &=\frac{\alpha^ 2b^*UV_{(i)}^{*}V_{(i)}U^{*}b}{\alpha^3 b^*UV_{(i)}^*}=\frac{1}{\alpha}V_{(i)}U^{*}b
\end{align*}
thus satisfying the required relation \eqref{eq:a2} for the solution.
\end{proof}

\textbf{Note-1 :} The cost of computation of $A^{*}b$ is $O(mn)$, and of $((b^{*}A)A^{*})A$ is also $O(mn)$ using three such matrix-vector products, and thus we can evaluate all components of the least square solution in $O(mn)$ computations using the relation in $\Cref{thm:direct_solver}$. Further, it is to be noted that by replacing $AA^*A^{(i)}$ with $A^{(i)}$ in \eqref{eq:a1} when the exact partial isometry holds, we get the solution $x=A^*b$, the least square solution for a partial isometric system that one can verify as $A(A^*b)=AA^*(Ax)=Ax$ given $AA^*A=A$. Whereas, the proposed solution applies to \emph{scaled} partial isometric systems with an unknown scaling factor as well.


In the following part of the paper, we continue to establish some additional properties of scaled partial isometries. Though the definition in \cref{sec:introduction} is complete and sufficient, these properties may be useful in identifying its application in a practical setting.

\begin{theorem}\label{thm:similarity}
Let \( A \in \mathbb{R}^{n \times n} \). Then, $A$ is similar to  $\mu Q$  for some $\mu> 0$  and an orthogonal  $Q$ iff $A$  is diagonalizable and all eigenvalues of A  have magnitude $\mu$.
\end{theorem} 
\begin{proof}
($\Rightarrow$) Suppose $A$ is similar to $\mu Q$ for some $\mu > 0$ and orthogonal $Q$, i.e., $A = P^{-1} (\mu Q) P$ for some invertible $P$. 

Since $Q$ is orthogonal, it is diagonalizable over $\mathbb{C}$, and so is $\mu Q$. Therefore, $A$ is diagonalizable. The eigenvalues of $Q$ all have modulus $1$, so the eigenvalues of $\mu Q$ (and hence of $A$) all have modulus $\mu$.

($\Leftarrow$) Suppose $A$ is diagonalizable and all its eigenvalues have modulus $\mu > 0$. Then $A$ is similar to a diagonal matrix $D = \mathrm{diag}(\lambda_1, \ldots, \lambda_n)$ with $|\lambda_i| = \mu$ for all $i$. Each $\lambda_i$ can be written as $\mu e^{i\theta_i}$. Over $\mathbb{R}$, complex eigenvalues occur in conjugate pairs, so $D$ is similar over $\mathbb{R}$ to a block-diagonal matrix with real blocks (for real eigenvalues) and $2 \times 2$ rotation blocks (for complex conjugate pairs), which is the canonical form of a real orthogonal matrix. Thus, $A$ is similar to $\mu Q$ for some real orthogonal $Q$.
\end{proof}

Next, we show that for any unitarily diagonalizable matrix whose non-zero eigenvalues are the same up to a multiplication of $\pm 1$, we can have a similar result as in \cref{thm:direct_solver} to evaluate the least-squares solution of the corresponding linear system. These kinds of matrices (and any of their unitary similarity) are also scaled partial isometries, because $A=cUDU^*$, with unitary matrices $U$ and the non-zero diagonal entries of $D \in \{\pm 1\}$. It is easy to see that $P=UDU^*$ is a partial isometry and thus $A=cP$ is a scaled partial isometry.

\begin{theorem}\label{thm:thm3}
    Let $A \in \mathbb{C}^{m \times m}$ be unitarily diagonalizable as $A= UDU^*$, with all the non-zero diagonal entries being the same up to a multiplication of $\pm 1$. Then, $A^\dagger=UD^\dagger U^*$.
\end{theorem}
\begin{proof}
    Let B=$UD^\dagger U^*$. We know $U^*U=I$ and $(DD^\dagger)^*=DD^\dagger$. Hence,
    \begin{align}
        & (AB)^*=(UDU^*UD^\dagger U^*)^*=U(D D^\dagger)^* U^*=UDD^\dagger U^*=AB. \label{eq:thm3_1}\\
        & (BA)^{*}=(UD^\dagger U^*UDU^*)^*=U(D^\dagger D)^*U^*=UD^\dagger DU^*=BA. \label{eq:thm3_2}\\
        & ABA=(UDU^*)(U D^\dagger U^*)(UDU^*)=UDD^\dagger DU^*=UDU^*=A,\label{eq:thm3_3} \text{ using $DD^\dagger D=D$.}\\
        & BAB=(UD^\dagger U^*)(UDU^*)(UD^\dagger U^*)=UD^\dagger DD^\dagger U^*= UD^\dagger U^*=B,\label{eq:thm3_4} \text{ using $D^\dagger D D^\dagger =D^\dagger$.}
    \end{align}
    From \eqref{eq:thm3_1}-\eqref{eq:thm3_4}, we know that $B$ is Moore-Penrose pseudoinverse of $A$. Thus, $A^\dagger = B = U D^\dagger U^*$.
\end{proof}

For such unitarily diagonalizble matrices $A=U_1DU_2^*$ (where$U_1,U_2$ are unitary), as mentioned in \cref{thm:thm3}, we can ensure that all the non-zero diagonal entries of $D$ have the real part as non-negative. For the instances, $Re(D_{i,i})<0$, one can flip the sign of $i^{th}$ column of $U_1$, and also of $i^{th}$ row of $D$ to obtain $\tilde{U}, \tilde{D}$ respectively, such that $U_1D=\tilde{U}\tilde{D}$. Now we can use \cref{thm:direct_solver} to evaluate the least-square solution $x_\star=U\tilde{D}^\dagger\tilde{U}^*b$ in $O(mn)$ computations for any such scaled systems (and their unitarily similar systems) in \cref{thm:thm3}.

\textbf{Note-2 :} For a linear system of equations represented as \( Ax = b \), if \( A \) is a block diagonal matrix and each block is a scaled partial isometry, we can still apply Theorem \ref{thm:direct_solver} to each block to obtain the minimum 2-norm solution. It is not necessary for all blocks to be scaled by the same scalar \( c \). Consequently, these types of matrices \( A \) may have different non-zero singular values and eigenvalues, and they are not necessarily the same up to a multiplication of \( \pm 1 \).

\section{Numerical Experiments }\label{sec:numerical_experiments}
 
For \Cref{tab:matrix_with_equal_singular_values}, we generated real matrices using \Cref{algo:matrix_with_equal_singular_values} with entries of $X$, $Y$ from a standard normal distribution. We generated $t \in \mathbb{R}^{n \times 1}$ with entries from a standard normal distribution and $b=At$. For \Cref{tab:matrix_with_equal_singular_values1}, we generated complex matrices by having $X=X_R + iX_I, Y=Y_R + iY_I$ in \Cref{algo:matrix_with_equal_singular_values}. We also generated $t \in \mathbb{C}^{n \times 1}$ through $t_R + it_I$ where we evaluate $b=At$. All random variables $X_R$, $X_I$, $Y_R$, $Y_I$, $t_R$, $t_I$ were independently drawn from a standard normal distribution. In both cases, the true solution $x_\star$ was obtained as $A^\dagger b$. For every fixed size, we averaged over 150 trials and reported the mean value of $\| x-x_\star \|_2$. All experiments have been executed in MATLAB 2024b on a computer with Intel(R) Core(TM) i9-14900K @ 6.0 GHz, 62.00 GB RAM, and 16.00 GB memory.

\Cref{tab:matrix_with_equal_singular_values}and \Cref{tab:matrix_with_equal_singular_values1} presents these experimental results demonstrating the validity of \Cref{thm:direct_solver}. We also know that the algorithm needs only three matrix-vector products in the evaluation of the true solution, which are numerically stable operations, and hence the method is numerically stable.

 \begin{table}
\caption{Validation of \Cref{thm:direct_solver} with \emph{real} matrices generated using  \Cref{algo:matrix_with_equal_singular_values}}\label{tab:matrix_with_equal_singular_values}
\begin{tabular*}{\tblwidth}{@{}LL@{}LL@{}LL@{}LL@{}}
\toprule
  s & 10 & 10 & 10 & 10 \\
        
  Size & 10000 $\times$ 10000 & 10000 $\times$ 2000  & 5000 $\times$ 10000 & 5000 $\times$ 5000\\
         
  r & 2000 & 400 & 2000 & 5000\\
\midrule
 $\|x-x_\star\|_2$  & 1.81e-13 & 3.34e-14  & 1.79e-13 & 1.97e-13\\
\bottomrule
\end{tabular*}
\end{table}

\begin{table}
\caption{Validation of \Cref{thm:direct_solver} with \emph{complex} matrices generated using  \Cref{algo:matrix_with_equal_singular_values}}\label{tab:matrix_with_equal_singular_values1}
\begin{tabular*}{\tblwidth}{@{}LL@{}LL@{}LL@{}LL@{}}
\toprule
  s & 10 & 10 & 10 & 10 \\
        
  Size & 10000 $\times$ 10000 & 10000 $\times$ 2000  & 5000 $\times$ 10000 & 5000 $\times$ 5000\\
         
  r & 2000 & 400 & 2000 & 5000\\
\midrule
 $\|x-x_\star\|_2$  &2.25e-13  &1.20e-13   &6.79e-15  &5.23e-14 \\
\bottomrule
\end{tabular*}
\end{table}

\begin{algorithm}
\caption{Generate Random Matrix with Equal Non-zero Singular Values}\label{algo:matrix_with_equal_singular_values}
\begin{algorithmic}[1]
\Require Integers $m, n, r$ with $r \leq \min(m,n)$, scalar $s$
\Ensure Matrix $A \in \mathbb{C}^{m \times n}$ with rank $r$ and all non-zero singular values equal to $s$
\State Generate a random matrix $X \in \mathbb{C}^{m \times m}$
\State Compute QR decomposition: $X = QR$, set $U \gets Q$
\State Generate a random matrix $Y \in \mathbb{C}^{n \times n}$
\State Compute QR decomposition: $Y = QR$, set $V \gets Q$
\State Initialize $\Sigma \in \mathbb{R}^{m \times n}$ as a zero matrix
\For{$i = 1$ to $r$}
    \State $\Sigma_{i,i} \gets s$
\EndFor
\State Compute $A \gets U \cdot \Sigma \cdot V^*$
\State \Return $A$
\end{algorithmic}
\end{algorithm}

\section{Conclusion}\label{sec:conclusion}
We presented a $O(mn)$ direct least-squares solver for the system of linear equations $Ax=b$ where A is a scaled partial isometry, and also where A is a block diagonal matrix with each block being a scaled partial isometry. Numerical experiments demonstrate the validity of the proposed method.  



\bibliographystyle{cas-model2-names}

\bibliography{reference}



\end{document}